\documentclass[12pt,reqno]{amsart}
\usepackage{amsaddr}   
\usepackage{lineno,marginnote}
\usepackage{comment}
\usepackage{xspace}
\usepackage[colorlinks = true,
			citecolor = blue,
			linkcolor = blue]{hyperref}
\usepackage[]{todonotes}			  
\usepackage{amssymb}
\usepackage{physics}
\usepackage{mathtools}
\mathtoolsset{showonlyrefs} 
\usepackage{booktabs}
\usepackage[] {datetime2}
\usepackage{enumerate}

\author[C. Khor]{Calvin Khor}
\address{School of Mathematical Sciences, Beijing Normal
University, \\ Beijing 100875, P. R. China.}
\email{C.Khor@bnu.edu.cn}

\author[J.L. Rodrigo]{Jos\'e L. Rodrigo}
\address[J.L. Rodrigo]
{Mathematics Research Centre,
Zeeman Building, \\
University of Warwick,
Coventry CV4 7AL,\\
United Kingdom}
\email{J.Rodrigo@warwick.ac.uk}

\newcommand{\del}{\partial}

\newcommand{\nrml}{N}

\newcommand{\mycurve}{z}

\newcommand{\dq}[5]
{
\frac{  {#1}({#4}) - {#1}({#5})          }
     {  |{#2}({#4}) - {#2}({#5})|^{#3}   }  
} 
\newcommand{\fd}{\mathcal D}

\makeatletter
\newcommand{\myitem}[1]{%
\item[#1]\protected@edef\@currentlabel{#1}%
}

\newtheorem{thm}{Theorem}[section]

\newtheorem{lem}[thm]{Lemma}
\newtheorem{prop}[thm]{Proposition}
\theoremstyle{definition}
\newtheorem{defn}[thm]{Definition}
\newtheorem{example}[thm]{Example}
\theoremstyle{remark}
\newtheorem{rem}[thm]{Remark}
\numberwithin{equation}{section}

\DeclareMathOperator{\mysin}{Sin} 
\newcommand{\FTC}[0]{\mathcal I}
\newcommand{\CK}{Cauchy--Kowalevski}
\newcommand{\myI}{\mathcal I}
\newcommand{\myIntTerm}[0]{\zeta} 
\newcommand{\myIBA}{\FTC(\mycurve_s)}
\newcommand{\myICA}{\FTC(\mycurve_{ss})}
\newcommand{\myIDA}{\FTC(\mycurve_{sss})}

\newcommand{\ACKu}{u}   

\allowdisplaybreaks

\usepackage{cite}

\begin{document}

\title[Local Existence of Analytic SFs for Singular SQG]{Local Existence of Analytic Sharp Fronts for Singular SQG  }
\date{\DTMnow}

\begin{abstract}
    In this paper, we prove local existence and uniqueness of analytic sharp-front solutions to a generalised SQG equation  by the use of an abstract Cauchy--Kowalevskaya theorem. Here, the velocity is determined by $u = |\nabla|^{-2\beta}\nabla^\perp\theta $ which (for $1<\beta\leq 2$) is more singular than in SQG.  This is achieved despite the appearance of pseudodifferential operators of order higher than one in our equation, by recasting our equation in a suitable integral form.  We also provide a full proof of the abstract version of the Cauchy--Kowalevskaya theorem we use.
\end{abstract}

\maketitle 

\setcounter{tocdepth}{1}
\tableofcontents
\section{Introduction}

The $\beta$-generalised SQG equation, $\beta\in(0,2]$ is the equation for the active scalar $\theta = \theta(x,t)\in\mathbb R,x\in\mathbb R^2,t\ge 0$,
\[ \theta_t + (|\nabla|^{-2+\beta}\nabla^\perp \theta )\cdot \nabla \theta = 0. \]
Here, $|\nabla|^{-2+\beta}$ is defined as the Fourier multiplier $|\xi|^{-2+\beta   }$.
Choosing $\beta = 0$  corresponds to the two-dimensional incompressible Euler equations in vorticity form, while $\beta=1$ corresponds to the Surface Quasi-Geostrophic (SQG) equation. In this paper, we consider special solutions to the family of equations for $\beta\in(1,2)$.  These solutions are of the form of an indicator of an evolving simply connected set with smooth boundary, termed sharp front solutions,
\[ \theta(x,t) = \mathbf 1_{x\in A(t)},\]
the evolution can be rephrased (by following the derivation in \cite{rodrigo2004vortex}) as a contour dynamics equation (CDE) for the boundary curve $z=\del A(t)$, namely 
\[ \mycurve_t(s,t)\cdot\nrml(s,t) = -C_\beta\int_{\mathbb T} \frac{(\mycurve_s(s_*,t) - \mycurve_s(s,t))\cdot\nrml(s,t) }{|\mycurve(s_*,t) - \mycurve(s,t)|^\beta} \dd{s_*},\]
where $\nrml$ is the unit normal to $\mycurve$, and $C_\beta$ should be the constant $ \frac{\Gamma(\beta/2)}{\pi 2^{2-\beta} \Gamma((2-\beta)/2)}$, 
 but for ease of presentation we will set $C_\beta := 1$.
The torus $\mathbb T$  is defined as $\mathbb T := \mathbb R/\mathbb Z$, and we will take $[0,1)$ as its fundamental domain in what follows.
This integral term is parameterisation independent, and also  finite  so long as the parameterisation is regular and sufficiently smooth, and the curve does not self-intersect. 
By a suitable reparameterisation of $s\in\mathbb T$, this is equivalent to the contour dynamics equation 
\begin{align} \mycurve_t(s,t) 
&= - \int_{\mathbb T} \frac{\mycurve_s(s_*,t) - \mycurve_s(s,t) }{|\mycurve(s_*,t) - \mycurve(s,t)|^\beta} \dd{s_*} + \lambda(s,t) \mycurve_s(s,t)\label{eqn-equiv-SF-system} \\
&=: \myIntTerm(s,t)+\lambda(s,t) \mycurve_s(s,t), \label{eqn-equiv-SF-system-short} \end{align} 
for some function $\lambda=\lambda(s,t)\in\mathbb R$, satisfying $\lambda(0,t)=0$ for every $t\ge 0$.  The integral term $\myIntTerm(s,t)$ is defined by \eqref{eqn-equiv-SF-system}. 

The corresponding CDE for the Euler equation with $\beta = 0$, termed the vortex patch problem, was first derived in \cite{zabusky197996} and its systematic study can be found in the book \cite{majda2002}. The study of sharp fronts for the SQG equation was initiated in \cite{rodrigo2005evolution}, where existence and uniqueness was proven in the class of smooth periodic graphs. Gancedo proved existence in \cite{gancedo2008existence} and later uniqueness \cite{gancedo2016arXiv160506663C} for Sobolev data and closed contours. The interpolated equations for $\beta \in (0,1)$ was further studied in \cite{gancedo2008existence},\cite{gancedo2018arXiv181100530G}, \cite{kiselevdoi:10.1002/cpa.21677}, and more recently the more singular sharp fronts ($\beta \in (1,2)$ as in this paper, or other variants of SQG) have been studied in \cite{chae2011inviscid},\cite{chae2012generalized},\cite{ohkitani2012asymptotics}.     Non-periodic graph-type sharp fronts of \eqref{eqn-singular-sqg} were studied in the recent papers, and  of Hunter, Shu, and Zhang. They gave two different approaches in \cite{hunter2018regularized} and \cite{hunter2019contour} to deriving the correct contour dynamics equation in this setting, and proved \cite{hunter2018global} global existence under a certain smallness condition. They also study two-front dynamics \cite{hunter2019twofront} and approximate equations for sharp fronts \cite{hunter2018local}. Finally, a second paper \cite{khor2019sfasf} by the authors has been prepared alongside this one that studies sharp fronts and almost-sharp fronts of the $\beta$-generalised SQG equation for a bounded domain in more detail.

In this paper, we are specifically interested in analytic sharp fronts, as in \cite{fefferman2011analytic}. The motivation is as follows. There are striking similarities of the 3D Euler equation with SQG ($\beta=1$)  (see \cite{CMT1, CMT2} for example),  that indicate that a sharp front is a good model of a vortex filament. Therefore, a key motivation for the study of the SQG sharp front is in developing methods that may transfer to the study of the vortex filaments. 

It is also natural to study smooth approximations to the sharp fronts. These approximations, with a fixed time of existence despite arbitrarily large gradient(corresponding to the thickness of the transition region)  have been called almost-sharp fronts, and have been studied for $\beta \in(0,1)$ in \cite{atkins2013almost}, and for SQG in  \cite{cordoba2004almost}, \cite{fefferman2012spine},\cite{fefferman2012almost},  \cite{fefferman2011analytic}, and \cite{fefferman2015construct}. These last three papers lay out a  proof strategy that may be possible to extend to the range $\beta\in(1,2)$.  Therefore this paper represents the first step in proving the existence of almost-sharp fronts, in the sense of \cite{fefferman2015construct}, for our more singular equation with $\beta \in (1,2)$. We are interested in analytic solutions, as opposed to the Sobolev regularity found in \cite{chae2012generalized}, since this programme introduces Prandtl-like terms that require the analytic setting. To this end, the main theorem of this paper is   
\begin{thm} Let $\mycurve_0=\mycurve_0(s):\mathbb T \to \mathbb R^2$ be an analytic, simple, closed curve. Then there exists $T>0$ and an analytic solution $\mycurve = \mycurve(s,t) : \mathbb T \times [0,T] \to \mathbb R^2$ to \eqref{eqn-equiv-SF-system} such that $\mycurve(s,0) = \mycurve_0(s)$. \label{main-thm}
\end{thm}

The remainder of this paper is organised as follows. In Section \ref{ch5section-prelim}, we define the basic objects of study and some basic geometric notation.  In Section \ref{ch5section-computing-ds}, we establish some formulas for derivatives  of $\myIntTerm$. In Section \ref{ch5section-rewriteEqn}, we carefully rewrite the evolution equations in a form suitable for the application of the abstract Cauchy--Kowalevski Theorem. In section \ref{ch5section-applyCK}, we state and prove a version of the abstract Cauchy--Kowalevski Theorem, and then we apply it to prove the main theorem (Theorem \ref{main-thm}) by using a suitable scale of Banach spaces.

\section{Preliminaries}\label{ch5section-prelim}
To obtain a useful expression for $\lambda$, we will further reparameterise $s$ so that 
\begin{align} 
|\mycurve_s(s,t)|^2=\mycurve_s(s,t) \cdot \mycurve_{s}(s,t)= L(t)^2 ,\quad  s\in \mathbb T, t\ge 0
\label{cancellation-from-parameterisation} .   
\end{align}

That is to say, $\mycurve_s \cdot \mycurve_{ss}$ vanishes identically. In this  parameterisation, $s$ is not arc-length, but the vectors $\mycurve_s(s,t)$ have  length $L(t)$ independent of  $s$. The quantity $L(t)$  is the  length of the curve $\mycurve$ at time $t$.
As in \cite{gancedo2008existence}, the function $\lambda$ can then be written explicitly in terms of $\mycurve$ as follows. Taking the $s$ derivative of $\mycurve_t$ in equation \eqref{eqn-equiv-SF-system-short} we find the equation 
\begin{align}
     \mycurve_{st} = \myIntTerm_s + \lambda\mycurve_{ss} + \lambda_s \mycurve_s. \label{eqn-mycurve_s-evo}
\end{align}
Therefore
\begin{align}\frac12 (L^2)'= L'L = \mycurve_{s}\cdot\mycurve_{st} = \mycurve_s \cdot \myIntTerm_s + \lambda_s L^2 \label{eqn-deriv-L} ,\end{align}
where we have substituted  \eqref{eqn-mycurve_s-evo} in the last equality. Now \eqref{eqn-deriv-L} yields
\begin{align}
    \frac{L'}L = \frac1{L^2} \mycurve_s \cdot \myIntTerm_s + \lambda_s.\label{eqn-intermediate-Lstep}
\end{align}
Since $\lambda$ is a periodic function in $s$, integrating \eqref{eqn-intermediate-Lstep} over $\mathbb T$, we obtain
 \begin{align}
      \frac{L'}L = \int_{\mathbb T}  \frac{\mycurve_s}{L^2} \cdot \myIntTerm_s \dd{s}. \label{eqn-l-prime-by-l}
 \end{align}

Hence, integrating \eqref{eqn-intermediate-Lstep} from $0$ to $s$ (using $\lambda(0,t)=0$) and dividing by $L^2$, we see that (using \eqref{eqn-l-prime-by-l})
\begin{align}
\lambda(s,t) 
&= s\frac{L'(t)}{L(t)} - 
\int_0^s \frac{\mycurve_{s}(s_1,t)}{L(t)^2}\cdot \myIntTerm_s(s_1,t)    \dd{s_1} \notag \\
&= -s  \int_{\mathbb T} \frac{\mycurve_{s}(s_1)}{|\mycurve_{s}(s_1)|^2}\cdot \del_{s_1} \int_\mathbb T \dq{\mycurve_s}{\mycurve}{\beta}{s_0}{s_1}   \dd{s_0} \dd{s_1 } 
\notag\\ 
&\quad  + \int_0^s \frac{\mycurve_{s}(s_1)}{|\mycurve_{s}(s_1)|^2 }\cdot \del_{s_1} \int_\mathbb T \dq{\mycurve_s}{\mycurve}{\beta}{s_0}{s_1}   \dd{s_0} \dd{s_1 }, \label{eqn-defn-lambda}
\end{align}
(hiding the $t$ dependence for legibility reasons) which is an expression for $\lambda$ completely determined by the values of $\mycurve$ and $\mycurve_s$. Conversely, if $\lambda$ is defined by \eqref{eqn-defn-lambda}, 
and $\mycurve$ evolves via \eqref{eqn-equiv-SF-system}, then $\lambda(0,t) = 0$ for every $t$, and differentiating \eqref{eqn-equiv-SF-system} readily yields
 \[ \frac12 \del_t |\mycurve_s|^2 = \myIntTerm_s \cdot \mycurve_s + \frac\lambda2 \del_s |\mycurve_s|^2 + \lambda_s |\mycurve_s|^2 = \frac\lambda2 \del_s |\mycurve_s|^2 + \mu(t) |\mycurve_s|^2,  \]
 where $\mu(t)$ is given by  \begin{align}
\mu(t)&=\lambda_s(s,t) + \frac{\myIntTerm_s(s,t)\cdot \mycurve_s(s,t) }{ |\mycurve_s (s,t) |^2 }  
\\& = -\int_{\mathbb T} \frac{\mycurve_{s}(s_1)}{|\mycurve_s(s_1)|^2}\cdot \del_{s_1} \int_\mathbb T \dq{\mycurve_s}{\mycurve}{\beta}{s_0}{s_1 }   \dd{s_0} \dd{s_1 }.      
 \end{align}
 One obtains by the method of characteristics for $F(s,t)= |\mycurve_s|^2(s,t)$, $F_0(s) = |\mycurve_s|^2(s,0)$,
 \[ F(s,t) = F_0\left (s+\int_0^t\lambda(s,\tau)\dd{\tau}\right )\exp\left(\int_0^t \mu(\tau) \dd{\tau}\right).  \]
 Hence, if $ \dv{F_0}{s}=0$, we have $\del_s F = 0$ for all times, and \eqref{eqn-mycurve_s-evo} is satisfied. This proves the following proposition:
 \begin{prop}[Parameterisation determines $\lambda$, \cite{gancedo2008existence}]
     Suppose that $\lambda(0,t)=0$, and that $\mycurve=\mycurve(s,t)$ is a smooth solution to the sharp front equation \eqref{eqn-equiv-SF-system} with initial condition $\mycurve_0=\mycurve_0(s)$ parameterised so that
     $\del_s \mycurve_0\cdot\del_s^2\mycurve_0 = 0$. Then $|\mycurve_s|^2$ is independent of $s$ for all $t$, or equivalently $\mycurve_s \cdot \mycurve_{ss} = 0$, if and only if  $\lambda$ is given by \eqref{eqn-defn-lambda}.
 \end{prop}
 
\section{Computing derivatives of terms  }\label{ch5section-computing-ds}
When computing derivatives in $s$, it is useful to use the periodicity in $s$ to rewrite the periodic integrands 
in terms of $ s_*+s$ and $s$ instead of $s_*$ and $s$. This is so that the derivatives of difference quotients like $\dq{a}{b}{\beta}{s+s_*}{s}$ retain their difference quotient structure. 

For simplicity we introduce notation for a finite difference,
\[ \fd a(s,s_*):=a(s+s_*) - a(s). \]

The relevant quantities (suppressing the time variable) become
\begin{align}
     \myIntTerm(s)&
		=-\int_{\mathbb T} \frac{\fd\mycurve_s(s,s_*)}{|\fd\mycurve(s,s_*)|^\beta} \dd{s_*}, \label{defn-I}\\
     \lambda(s) &= s  \int_{\mathbb T} \frac{\mycurve_{s}(s_1)}{L^2}
     \cdot
       \myIntTerm_s(s_1) \dd{s_1 } 
 - \int_0^s \frac{\mycurve_{s}(s_1)}{L^2}
     \cdot
      \myIntTerm_s(s_1) \dd{s_1 } . \label{defn-lambda} \end{align}

Note that $\zeta $ is smooth if $\mycurve$ is,  by differentiating under the integral.
Therefore, the first two derivatives of $\myIntTerm$ are (writing $\fd h$ for $\fd h(s,s_*)$) 
\newcommand{\Df}{\fd\mycurve_s}
\newcommand{\Dfprime}{\fd\mycurve_{ss}}
\newcommand{\Dfpprime}{\fd\mycurve_{sss}}
\newcommand{\Dfppprime}{\fd\mycurve_{ssss}}
 
\newcommand{\Dg}{\fd\mycurve}
\newcommand{\Dgprime}{(\fd\mycurve\cdot\fd\mycurve_s)}

\newcommand{\DgpprimeA}{(|\fd\mycurve_s|^2 + \fd\mycurve\cdot\fd\mycurve_{ss})   }
\newcommand{\DgpprimeB}{(\fd\mycurve\cdot\fd \mycurve_s)^2}

\begin{align}
\myIntTerm_s(s)&=
-\int_{\mathbb T} 
\frac{\fd\mycurve_{ss}}{|\fd\mycurve|^{\beta}}
    -\beta\frac{\fd \mycurve_s(\fd \mycurve\cdot \fd\mycurve_{s})}{|\fd \mycurve|^{\beta+2}}\dd{s_*}, \label{eqn-1st-zeta-deriv} \\
\myIntTerm_{ss}(s) &= 
-\int_{\mathbb T} \Bigg(
\frac{\Dfpprime}{|\Dg|^\beta} 
- 2\beta\frac{\Dfprime \Dgprime}{|\Dg|^{\beta+2}} \notag\\ 
&\quad -\beta \frac{\Df \DgpprimeA}{|\Dg|^{\beta+2}}\notag\\
&\quad +\beta(\beta+2)\frac{\Df\DgpprimeB}{|\Dg|^{\beta+4}} \bigg)\dd{s_*},   \label{eqn-2nd-zeta-deriv}
\end{align}

which yields 
\begin{align}
\lambda_s(s) &=  
     \int_{\mathbb T} \frac{\mycurve_{s}(s_1)}{L^2}
     \cdot
      \myIntTerm_s(s_1) \dd{s_1 } 
  -\frac{\mycurve_{s}(s)}{L^2}
     \cdot
      \myIntTerm_s(s), \label{eqn-lambda-deriv} \\ 
\lambda_{ss} (s) &=  -\frac{ \mycurve_s(s)\cdot \myIntTerm_{ss}(s) }{L^2} - \frac{ \mycurve_{ss}(s)\cdot \myIntTerm_s(s)}{L^2} . \label{eqn-lambda-dderiv}
\end{align}
We have computed all the terms on the right hand side of  the evolution equations
\begin{align}
\label{eqn-evo1} \mycurve_t &= \myIntTerm + \lambda \mycurve_s,\\
\label{eqn-evo2}\mycurve_{st} &= \myIntTerm_s + \lambda_s \mycurve_s + \lambda \mycurve_{ss}, \\
\label{eqn-evo3}\mycurve_{sst} &= \myIntTerm_{ss} + \lambda_{ss} \mycurve_s + 2\lambda_s\mycurve_{ss} + \lambda \mycurve_{sss}. 
\end{align}

As mentioned earlier, we also need to assume that the initial curve $\mycurve_0$ does not self-intersect so that the integral $\myIntTerm(s,t)$ makes sense. In analogy with  \cite{gancedo2008existence},
let us introduce the arc-chord condition via the function $\Gamma=\Gamma(\mycurve)$ below. It acts as a quantitative control on the length of the curve (since $\Gamma(0,0)=\frac1L$) and self-intersection, and is a slight variant of the function $F$ introduced in \cite{gancedo2008existence}.

\begin{defn}[Arc-chord condition] \label{defn-ac-cond} We say that a curve $\mycurve$ satisfies the arc chord condition if the function $ \Gamma(z) : \mathbb T^2 \to \mathbb R$ \label{defn-gamma} defined by
\begin{align}
 \Gamma (\mycurve)(s,s_*) :=\begin{cases}
\dfrac{|\sin(\pi s_*)|}{\pi |\mycurve(s+s_*)-\mycurve(s)|} & s_*\neq 0,\\ \dfrac{1}{|\mycurve_s(s)|} &s_* = 0,    \label{eqn-defn-gamma}
\end{cases}    
\end{align}
belongs to $L^\infty(\mathbb T^2)$.
\end{defn}
\begin{lem}[Analyticity of $\Gamma $]
Suppose $\mycurve$ is an analytic curve with uniform speed parameterisation $s\in\mathbb T$, and $\Gamma(\mycurve)\in L^\infty(\mathbb T^2) $. Then $\mycurve$ doesn't self-intersect, and $\Gamma(\mycurve)$ is  analytic on $\mathbb T^2$. In addition, $\Gamma(\mycurve)$ is bounded away from $0$.
\end{lem}
\begin{proof}
Notice that $Z:=\frac{\mycurve(s+s_*) - \mycurve(s)}{\sin(\pi s_*)/\pi }$ is analytic and does not vanish. The Euclidean norm of $Z$ is therefore analytic, and the fact that this avoids zero since $\Gamma(\mycurve) \in L^\infty$  means that its reciprocal is also analytic.
\end{proof}
To simplify the presentation slightly, we define $ \mysin s := \sin(\pi s)/\pi. $

The goal is to apply a \CK{} type argument, as in \cite{fefferman2011analytic}. The contour dynamics equation involves a term of order $1+\beta$, which means that a direct application of Cauchy-Kowaleski is not an option. We will need to carefully write the evolution equations for
\[ \mycurve,\mycurve_s,\mycurve_{ss},\text{ and } \Gamma,   \]
in terms of analytic functions involving at most first order operators of those quantities plus one higher order term, that has a linear structure. To this end, we will use the following elementary version of the Fundamental Theorem of Calculus.
\begin{lem}
     If $z\in C^1$ is a curve, $s,s_*\in\mathbb T$, then if we define for curves $w$,
\begin{align}
    \FTC(w)(s,s_*) := \int_0^1 w(s+(1-\tau)s_*) \dd{\tau}, \label{eqn-FTC}      
\end{align}
then
\begin{align}
    \fd \mycurve(s,s_*) &= \FTC(\mycurve_s)(s,s_*)s_*.
\end{align}
\end{lem}

It will be convenient to extend the definition of $\FTC$ to two-variable functions $F=F(s,s_*)$ by the formula
\begin{align} \FTC(F)(s,s_*) := \int_0^1 F(s,(1-\tau)s_*) \dd{\tau} \label{eqn-extnded-I0}.
\end{align}
This agrees with the previous definition \eqref{eqn-FTC} in the sense that if $ F(s,s_*) = z(s+s_*),$ then $\FTC (F)(s,s_*) = \FTC (z)(s,s_*)$. 
\begin{prop}[Expansion of $\Gamma$] The function $\Gamma : \mathbb T^2 \to \mathbb R$ as defined in \eqref{eqn-defn-gamma} satisfies the following first order expansion in $s_*$,
    \[ \Gamma(s,s_*,t) = \frac1{L(t)} + \FTC(\del_{s_*}\Gamma)(s,s_*)s_*, \]
    where explicitly,  if $s_*\neq 0$ (ignoring the $t$ variable),
    \[ \del_{s_*} \Gamma(s,s_*) = \begin{cases}
     -\Gamma(s,s_*)^3 \frac{\fd \mycurve}{\mysin s_*   }\cdot\left(\frac{ \mycurve_s(s+s_*)  \mysin s_*  - \fd\mycurve(s,s_*)  \cos(\pi s_*) }{(\mysin s_* )^2} \right) & s_* \neq 0, \\ 0 & s_* = 0,
 \end{cases}
   \]
and using \eqref{eqn-extnded-I0}, we have $ \FTC(\del_{s_*}\Gamma)(s,s_*) := \int_0^1 (\del_{s_*}\Gamma)(s,(1-\tau)s_*) \dd{\tau}.  $
\end{prop}
\begin{proof}
    The formula for $\FTC(\del_{s_*} \Gamma)$ comes from applying the FTC to expand $\tilde \Gamma(s_*) := \Gamma(s,s_*)$ around $s_*=0$ for  fixed $s$,
    \[ \tilde\Gamma (s_*) =  \tilde\Gamma(0) + \int_0^{s_*} \tilde \Gamma'(s_*-\tilde \tau)\dd{\tilde  \tau} =  \tilde\Gamma(0) + \int_0^1 \tilde \Gamma'((1-\tau)s_*)\dd{\tau}s_*.\]
    Since $\tilde\Gamma(0) = \frac1{L}$, we just need to compute $\tilde\Gamma'$. We have
    \begin{align} \tilde\Gamma'(s_*) 
    &= \del_{s_*} \Gamma(s,s_*) = \del_{s_*}\left( \left|\frac{\fd\mycurve (s,s_*)}{\mysin s_*} \right|^{-1}  \right)
    \label{dgammaL1} \\
    &= -\left|\frac{\fd\mycurve (s,s_*)}{\mysin s_*}\right|^{-3}\frac{\fd\mycurve (s,s_*)}{\mysin s_*}\cdot \left(\del_{s_*} \frac{\fd\mycurve (s,s_*)}{\mysin s_*}\right)
    \label{dgammaL2}\\
    &= -\Gamma(s,s_*)^3  \frac{\fd\mycurve (s,s_*)}{\mysin s_*}\cdot \frac{(\mysin s_* )\mycurve_s (s+s_*)-\cos(\pi s_*)  \fd \mycurve(s,s_*)}{(\mysin s_*)^2} ,
    \label{dgammaL3}
    \end{align}
as claimed. For the behavior as $s_*\to0$, writing $\mysin s_* = s_* + O(s_*^3)$, and $\cos(\pi s_*) = 1-\pi^2 s_*^2 + O(s_*^4)$, 
\begin{align}
    &\del_{s*}\Gamma(s,s_*) \\
    &= -\Gamma(s,0)^3 \mycurve_{s}(s)\cdot  \frac{\mycurve_s(s+s_*)s_*-(1-\pi^2 s_*^2/2)\fd \mycurve(s,s_*)}{s_*^2 } + o(1)
    \label{dgammaL4} \\
    &= -\Gamma(s,0)^3 \mycurve_{s}(s)\cdot  \frac{\mycurve_s(s+s_*)s_*-(\mycurve_s(s)s_* + \mycurve_{ss}(s)s_*^2/2 + o(s_*^2) )}{s_*^2 } + o(1)
    \label{dgammaL5} \\
    &= -\Gamma(s,0)^3 \mycurve_s(s) \cdot 
    \frac{\fd{\mycurve_s }(s,s_*)s_* - \mycurve_{ss}(s)s_*^2/2 }{s_*^2} + o(1) \\
    &= -\Gamma(s,0)^3 \mycurve_s (s) \cdot \left(\mycurve_{ss}(s) - \frac {\mycurve_{ss}(s)}2 + o(1) \right )  + o(1)
    \label{dgammaL6}\\
    &= o(1) \label{dgamma7} , 
\end{align}
where in going from \eqref{dgammaL4} to \eqref{dgammaL5} we used a Taylor expansion of $\mycurve$ to rewrite $\fd\mycurve$, and at line \eqref{dgammaL6} we used $\mycurve_s \cdot\mycurve_{ss} = 0$.
\end{proof} An analogous calculation gives the following.
\begin{lem}[Expansion of $\Gamma^\beta$] \label{lem-gamma-beta}
\[\Gamma^\beta(s,s_*) = \frac1{L(t)^\beta} +  \FTC(\beta \Gamma^{\beta-1}\del_{s*}\Gamma )(s,s_*). \]
    
\end{lem}

 In rewriting the evolution equations \eqref{eqn-evo1}, \eqref{eqn-evo2}, \eqref{eqn-evo3}, the following dot products appear:\begin{align}
    \fd\mycurve\cdot\fd\mycurve_s  
    &= \myIBA \cdot \myICA s_*^2, \label{eqn-dotprod1} \\
|\fd\mycurve_s|^2 &= |\myICA|^2  s_*^2, \label{eqn-dotprod2}\\ 
\fd\mycurve\cdot\fd\mycurve_{ss} &= \myIBA\cdot\myIDA s_*^2.   \label{eqn-dotprod3}
\end{align}

%
%
%
%

We will also need expressions for $\zeta, \zeta_s$ and $\zeta_{ss}$.
Starting from \eqref{defn-I}, \eqref{eqn-1st-zeta-deriv}, and \eqref{eqn-2nd-zeta-deriv} we replace the finite differences $\fd \del_s^k\mycurve $ with the integral terms $ \FTC(\del_s^{k+1} \mycurve)s_*$ via \eqref{eqn-dotprod1}, \eqref{eqn-dotprod2}, and \eqref{eqn-dotprod3}. Then we rewrite $|\fd \mycurve(s,s_*)|^{-1}=|\mysin s_*|^{-1} \Gamma(s,s_*)$. 
We obtain the following expressions (already involving 3rd derivatives of $z$)
\begin{align} \myIntTerm(s)
&=  -\int_{\mathbb T} \Gamma(s,s_*)^\beta \frac{\myICA s_*}{|\mysin s_*|^\beta } \dd{s_*}, 
\label{eqn-I} \\
\myIntTerm_s(s)
&= -\int_{\mathbb T} \Gamma(s,s_*)^\beta \frac{\myIDA s_*}  {|\mysin s_*|^\beta}  \dd{s_*}, 
\notag \\ 
& \quad +\beta  \int_{\mathbb T}\Gamma(s,s_*)^{\beta+2}\frac{\myICA(\myIBA\cdot \myICA)s_*^3}{|\mysin s_* |^{\beta+2}} \dd{s_*}.
\label{eqn-Is} 
\end{align}
For the highest derivative $\zeta_{ss}$,  we additionally use the expansion of $\Gamma^\beta$ in Lemma \ref{lem-gamma-beta} to isolate a convolution operator for the most singular term.
\begin{align}
\myIntTerm_{ss}(s) & = -\frac1{L^\beta} \int_{\mathbb T} \frac{\mycurve_{sss} (s+s_*) -  \mycurve_{sss}(s) }{|\mysin s_*|^\beta } \dd{s_*} 
\notag \\
&-\int_{\mathbb T} \FTC(\beta \Gamma^{\beta-1} \del_{s_*}\Gamma  )(s,s_*)s_*\frac{\mycurve_{sss} (s+s_*) -  \mycurve_{sss}(s) }{|\mysin s_*|^\beta } \dd{s_*} 
\notag 
\\&+ 2\beta\int_{\mathbb T} \Gamma(s,s_*)^{\beta+2} \frac{\myIDA (\myIBA\cdot \myICA)s_*^3}{|\mysin s_* |^{\beta+2}} \dd{s_* }\notag\\ 
&+ \beta\int_{\mathbb T} \Gamma(s,s_*)^{\beta+2} \frac{\myICA (|\myICA|^2+\myIBA \cdot \myIDA )s_*^3}{|\mysin s_* |^{\beta+2}} \dd{s_*} \notag\\
&-\beta(\beta+2)\int_{\mathbb T} \Gamma(s,s_*)^{\beta+4} \frac{ \FTC(\mycurve_{ss}) (\myIBA\cdot \myICA)^2 s_*^5 }{|\mysin s_*|^{\beta+4}} \dd{s_*} .
\label{eqn-Iss} 
\end{align}

To rewrite $\lambda$ we start by considering $z_s\cdot \zeta_s$ (see \eqref{defn-lambda}).
Using \eqref{eqn-Is} we find
\begin{align}
    \mycurve_s \cdot \myIntTerm_s 
    &= \int_{\mathbb T}-\Gamma(s,s_*)^\beta \frac{\myIDA\cdot \mycurve_s  s_*}  {|\mysin s_*|^\beta} \\
    &\qquad     +\beta  \Gamma(s,s_*)^{\beta+2}\frac{\myICA\cdot \mycurve_s(\myIBA\cdot \myICA)s_*^3}{|\mysin s_* |^{\beta+2}} \dd{s_*}.  
\end{align}
Notice that this already includes $z_{sss}$.  However, notice that 
\begin{align*}
 \myIDA \cdot \mycurve _s(s) s_* 
 &= \mycurve_{ss}(s+s_*) \cdot \mycurve_s(s) 
\\&=  \mycurve_{ss}(s+s_*)\cdot(\mycurve_s(s+s_*) - \myICA s_*)\\
&= -\mycurve_{ss}(s+s_*) \cdot \myICA s_*,
\end{align*}
so we can remove completely the dependence on the third derivative: \begin{align}
\mycurve_s \cdot \myIntTerm_s 
&
= \int_{\mathbb T}\Gamma(s,s_*)^\beta \frac{\myICA\cdot \mycurve_{ss}(s+s_*)  s_*}  {|\mysin s_*|^\beta} \\
&\qquad +\beta  \Gamma(s,s_*)^{\beta+2}\frac{\myICA\cdot \mycurve_s(\myIBA\cdot \myICA)s_*^3}{|\mysin s_* |^{\beta+2}} \dd{s_*}.
\label{zscdotzetas}      
\end{align}
Hence, we can write  $\lambda$ as (using $\frac1{L} = \Gamma(s,0) = \Gamma(0,0)$)
\begin{align}
&  \lambda(s) \\&= 
s  \int_{\mathbb T} \frac{\mycurve_{s}(s_1)}{L^2}
     \cdot
       \myIntTerm_s(s_1) \dd{s_1 } 
 - \int_0^s \frac{\mycurve_{s}(s_1)}{L^2}
     \cdot
      \myIntTerm_s(s_1) \dd{s_1 }      
      \notag \\
      &=  \frac {s}{L^2}\int_{\mathbb T}\int_{\mathbb T} \bigg(\Gamma(s_1,s_0)^\beta \frac{\myICA(s_1,s_0)\cdot \mycurve_{ss}(s_0+s_1)  s_0}  {|\mysin s_0|^\beta} +\beta  \Gamma(s_1,s_0)^{\beta+2}\times 
      \notag \\
      &\quad  \frac{\myICA(s_1,s_0)\cdot \mycurve_s(s_1)(\myIBA(s_1,s_0)\cdot \myICA(s_1,s_0))s_0^3}{|\mysin s_0 |^{\beta+2}} \bigg)\dd{s_0}\dd{s_1} 
      \notag \\  
      &\quad -\frac1{L^2} \int_{0}^s\int_{\mathbb T} \bigg( \Gamma(s_1,s_0)^\beta \frac{\myICA(s_1,s_0)\cdot \mycurve_{ss}(s_0+s_1)  s_0}  {|\mysin s_0|^\beta} 
      + \beta  \Gamma(s_1,s_0)^{\beta+2}\times  
      \notag \\
      &\quad\frac{\myICA(s_1,s_0)\cdot \mycurve_s(s_1)(\myIBA(s_1,s_0)\cdot \myICA(s_1,s_0))s_0^3}{|\mysin s_0 |^{\beta+2}} \bigg) \dd{s_0}\dd{s_1}  . \label{eqn-lambda-v2}
\end{align}
(Here  $\times$ denotes multiplication.) Similarly, we can rewrite \eqref{eqn-lambda-deriv}  for $\lambda_s$ as (using \eqref{zscdotzetas})
\begin{align}
    &\lambda_s(s)
    \\ &= \frac{1}{L^2}\int_{\mathbb T}\int_{\mathbb T}\bigg( \Gamma(s_1,s_0)^\beta \frac{\myICA(s_1,s_0)\cdot \mycurve_{ss}(s_0+s_1)  s_0}  {|\mysin s_0|^\beta} 
    +\beta  \Gamma(s_1,s_0)^{\beta+2}\times
    \notag \\
      &
      \\
      &\quad \frac{\myICA(s_1,s_0)\cdot \mycurve_s(s_1)(\myIBA(s_1,s_0)\cdot \myICA(s_1,s_0))s_0^3}{|\mysin s_0 |^{\beta+2}} \bigg) \dd{s_0}\dd{s_1}
    \notag \\
      &\quad-\frac1{L^2}\int_{\mathbb T}\bigg( \Gamma(s,s_*)^\beta \frac{\myICA(s,s_*)\cdot \mycurve_{ss}(s+s_*)  s_*}  {|\mysin s_*|^\beta} 
      +\beta  \Gamma(s,s_*)^{\beta+2} \times
      \notag\\
      &\quad \frac{\myICA(s,s_*)\cdot \mycurve_s(s)(\myIBA(s,s_*)\cdot \myICA(s,s_*))s_*^3}{|\mysin s_* |^{\beta+2}}\bigg) \dd{s_*}.\label{eqn-lambda-deriv-v2}
\end{align}
Finally for $\lambda_{ss}$ recall that 
\begin{align*}
\lambda_{ss} (s) &= - \frac1{L^2} ( \myIntTerm_{ss}\cdot \mycurve_s   + \myIntTerm_s \cdot \mycurve_{ss})  .   
\end{align*}
 $\myIntTerm_{ss}$ involves an operator of order higher than 3 on $z$. However we have $\zeta_{ss} \cdot z_s$; using the identities (which follow from $\mycurve_s \cdot \mycurve_{ss} =  0$),
\begin{align*}
    \mycurve_{sss}(s)\cdot \mycurve_{s}(s) 
    &= - |\mycurve_{ss}(s)|^2,\\    
    \mycurve_{sss}(s+s_*)\cdot \mycurve_s(s) 
    &=  - |\mycurve_{ss}(s+s_*)|^2 -\mycurve_{sss}(s+s_*) \cdot \myICA(s,s_*) s_* , 
  \end{align*}
we can rewrite 
\begin{align*} 
&\fd \mycurve_{sss}(s,s_*) \cdot \mycurve_s(s) \\
&= - \fd |\mycurve_{ss}|^2(s,s_*) -\mycurve_{sss}(s+s_*)\cdot\myICA(s,s_*) s_* \\
&= \big( - 2 \FTC(\mycurve_{sss}\cdot \mycurve_{ss})(s,s_*) - \mycurve_{sss}(s+s_*) \cdot \myICA(s,s_*) \big )s_*. 
\end{align*}
Therefore, $\lambda_{ss}$ depends  on $\Gamma$ and the first three derivatives of $\mycurve$. For reference, the full expansion of $\lambda_{ss}$ is as follows, which comes from a similar derivation to that of \eqref{eqn-Iss} but the cancellations above are used to regularise the integral, in place of Lemma \ref{lem-gamma-beta}.
\begin{align}
\lambda_{ss}(s) & = \frac1{L^2} \int_{\mathbb T}  
\Gamma(s,s_*)^{\beta} \frac{(-2\FTC(\mycurve_{sss}\cdot \mycurve_{ss}) - \mycurve_{sss}(s+s_*) \cdot \myICA )s_*}{|\mysin s_*|^\beta } 
\\
&\quad  - 2\beta \Gamma(s,s_*)^{\beta+2} \frac{\myIDA\cdot\mycurve_{s} (\myIBA\cdot \myICA)s_*^3}{|\mysin s_* |^{\beta+2}} 
\notag\\ 
&\quad - \beta \Gamma(s,s_*)^{\beta+2} \frac{\myICA\cdot\mycurve_{s} (|\myICA|^2+\myIBA \cdot \myIDA )s_*^3}{|\mysin s_* |^{\beta+2}} 
\notag\\
&\quad +\beta(\beta+2) \Gamma(s,s_*)^{\beta+4} \frac{\myICA \  \cdot\mycurve_{s} (\myIBA\cdot \myICA)^2 s_*^5 }{|\mysin s_*|^{\beta+4}} 
\notag\\
&\quad + \Gamma(s,s_*)^\beta \frac{\myIDA\cdot \mycurve_{ss} s_*}  {|\mysin s_*|^\beta} 
\notag\\
&\quad -\beta  \Gamma(s,s_*)^{\beta+2}\frac{\myICA\cdot\mycurve_{ss} (\myIBA\cdot \myICA)s_*^3}{|\mysin s_* |^{\beta+2}} \dd{s_*}.
\label{eqn-lambda-dderiv-v2}
\end{align}
\section{Rewriting the evolution equations}
\label{ch5section-rewriteEqn}
We will consider $z$, $z_s$ and $z_ss$ as independent. We define
 \[ f(s,t)=\mycurve(s,t),\quad g(s,t) =\mycurve_s(s,t),\quad h(s,t) =\mycurve_{ss} (s,t).  \]
\subsection{Evolution of $f$}
Rewriting 
\eqref{eqn-evo1} using \eqref{eqn-I} and \eqref{eqn-lambda-v2}, we obtain
\begin{align}
& f_t(s)
\\ &= -\int_{\mathbb T} \Gamma(s,s_*)^\beta \frac{\FTC(h)(s,s_*)  s_*}{|\mysin s_*|^\beta } \dd{s_*} \notag \\
  &+ g(s) \Bigg( s \frac1{L^2} \int_{\mathbb T}\int_{\mathbb T} \Gamma(s_1,s_0)^\beta \frac{\FTC(h)(s_1,s_0)\cdot h(s_0+s_1)  s_0}  {|\mysin s_0|^\beta} \notag \\
      &+\beta  \Gamma(s_1,s_0)^{\beta+2}
			\frac{\FTC(h)(s_1,s_0)\cdot g(s_1)(\FTC(g)(s_1,s_0)\cdot \FTC(h)(s_1,s_0))s_0^3}{|\mysin s_0 |^{\beta+2}} \dd{s_0}\dd{s_1} 
      \notag \\  
      &-\frac1{L^2} \int_{0}^s\int_{\mathbb T} \Gamma(s_1,s_0)^\beta \frac{\FTC(h)(s_1,s_0)\cdot h(s_0+ s_1)  s_0}  {|\mysin s_0|^\beta} \notag \\
      &+\beta  \Gamma(s_1,s_0)^{\beta+2}
			\frac{\FTC(h)(s_1,s_0)\cdot g(s_1)(\FTC(g)(s_1,s_0)\cdot \FTC(h)(s_1,s_0))s_0^3}{|\mysin s_0 |^{\beta+2}} \dd{s_0}\dd{s_1}   \Bigg ).     \label{eqn-f}
\end{align}

\subsection{Evolution of $g$}
Rewriting  $\mycurve_{st} = \myIntTerm_s + \lambda \mycurve_{ss} + \lambda_s \mycurve_s$ \eqref{eqn-evo2} using \eqref{eqn-Is}, \eqref{eqn-lambda-v2}, and \eqref{eqn-lambda-deriv-v2}, we obtain
\begin{align}
g_t(s) &= -\int_{\mathbb T} \Gamma(s,s_*)^\beta \frac{\FTC(h_s) s_*}  {|\mysin s_*|^\beta}  \dd{s_*} \notag \\ 
&  +\beta  \int_{\mathbb T}\Gamma(s,s_*)^{\beta+2}\frac{\FTC(h)(\FTC(g)\cdot \FTC(h))s_*^3}{|\mysin s_* |^{\beta+2}} \dd{s_*}\notag \\
&+ h(s) \Bigg( s \frac1{L^2} \int_{\mathbb T}\int_{\mathbb T} \Gamma(s_1,s_0)^\beta \frac{\FTC(h)(s_1,s_0)\cdot h(s_0+s_1)  s_0}  {|\mysin s_0|^\beta} \notag \\
      &+\beta  \Gamma(s_1,s_0)^{\beta+2}
      \\ & \times \frac{\FTC(h)(s_1,s_0)\cdot g(s_1)(\FTC(g)(s_1,s_0)\cdot \FTC(h)(s_1,s_0))s_0^3}{|\mysin s_0 |^{\beta+2}} \dd{s_0}\dd{s_1} 
      \notag \\
        &-\frac1{L^2} \int_{0}^s\int_{\mathbb T} \Gamma(s_1,s_0)^\beta \frac{\FTC(h)(s_1,s_0)\cdot h(s_0+s_1)  s_0}  {|\mysin s_0|^\beta} \notag \\
      &+\beta  \Gamma(s_1,s_0)^{\beta+2}
      \\ & \times \frac{\FTC(h)(s_1,s_0)\cdot g(s_1)(\FTC(g)(s_1,s_0)\cdot \FTC(h)(s_1,s_0))s_0^3}{|\mysin s_0 |^{\beta+2}} \dd{s_0}\dd{s_1}   \Bigg ) \notag\\
      &+ g(s) \Bigg( \frac1{L^2} \int_{\mathbb T}\int_{\mathbb T} \Gamma(s_1,s_0)^\beta \frac{\FTC(h)(s_1,s_0)\cdot h(s_0+s_1)  s_0}  {|\mysin s_0|^\beta} \notag \\
      &+\beta  \Gamma(s_1,s_0)^{\beta+2} \\ 
      & \times \frac{\FTC(h)(s_1,s_0)\cdot g(s_1)(\FTC(g)(s_1,s_0)\cdot \FTC(h)(s_1,s_0))s_0^3}{|\mysin s_0 |^{\beta+2}} \dd{s_0}\dd{s_1}\notag   \\ 
      &-\frac1{L^2} \int_{\mathbb T} \Gamma(s,s_*)^\beta \frac{\FTC(h)(s,s_*)\cdot h(s+s_*)  s_*}  {|\mysin s_*|^\beta} \notag \\
      &+\beta  \Gamma(s,s_*)^{\beta+2}
			\frac{\FTC(h)(s,s_*)\cdot g(s)(\FTC(g)(s,s_*)\cdot \FTC(h)(s,s_*))s_*^3}{|\mysin s_* |^{\beta+2}} \dd{s_*}\Bigg) \label{eqn-g}.
 \end{align}

\subsection{Evolution of $h$}
Rewriting $\mycurve_{sst} = \myIntTerm_{ss} + \lambda_{ss} \mycurve_s + 2\lambda_{s} \mycurve_{ss} + \lambda \mycurve_{sss} $ \eqref{eqn-evo3} using \eqref{eqn-Iss}, \eqref{eqn-lambda-v2}, \eqref{eqn-lambda-deriv-v2}, and \eqref{eqn-lambda-dderiv-v2}, we obtain
\begin{align}
    h_t & =\frac{-1}{L^\beta} \int_{\mathbb T} \frac{h_s (s+s_*) -  h_s(s) }{|\mysin s_*|^\beta } \dd{s_*} \notag \\
&-\int_{\mathbb T} \FTC(\beta \Gamma^{\beta-1} \del_{s_*}\Gamma  )(s,s_*)s_*\frac{h_s (s+s_*) -  h_s(s) }{|\mysin s_*|^\beta } \dd{s_*} \notag  \\
&+ 2\beta\int_{\mathbb T} \Gamma(s,s_*)^{\beta+2} \frac{\FTC(h_s) (\FTC(g)\cdot \FTC(h))s_*^3}{|\mysin s_* |^{\beta+2}} \dd{s_* }\notag\\ 
&+ \beta\int_{\mathbb T} \Gamma(s,s_*)^{\beta+2} \frac{\FTC(h) (|\FTC(h)|^2+\FTC(g) \cdot \FTC(h_s) )s_*^3}{|\mysin s_* |^{\beta+2}} \dd{s_*} \notag\\
&-\beta(\beta+2)\int_{\mathbb T} \Gamma(s,s_*)^{\beta+4} \frac{\myI_{2,0 } (\FTC(g)\cdot \FTC(h))^2 s_*^5 }{|\mysin s_*|^{\beta+4}} \dd{s_*} 
\notag \\
& + g(s)\Bigg( \frac1{L^2}  \int_{\mathbb T}  \Gamma(s,s_*)^{\beta} \frac{(-2\FTC(h_s\cdot h) - h_s(s+s_*) \cdot \FTC(h) )s_*}{|\mysin s_*|^\beta } \notag \\
& - 2\beta \Gamma(s,s_*)^{\beta+2} \frac{\FTC(h)\cdot g (\FTC(g)\cdot \FTC(h))s_*^3}{|\mysin s_* |^{\beta+2}} \notag\\ 
&- \beta \Gamma(s,s_*)^{\beta+2} \frac{\FTC(h)\cdot g (|\FTC(h)|^2+\FTC(g) \cdot \FTC(h_s) )s_*^3}{|\mysin s_* |^{\beta+2}} \notag\\
&+\beta(\beta+2) \Gamma(s,s_*)^{\beta+4} \frac{\myI_{2,0 }\cdot g (\FTC(g)\cdot \FTC(h))^2 s_*^5 }{|\mysin s_*|^{\beta+4}} \notag \\
&+ \Gamma(s,s_*)^\beta \frac{\FTC(h_s)\cdot h s_*}  {|\mysin s_*|^\beta}   -\beta  \Gamma(s,s_*)^{\beta+2}\frac{\FTC(h)\cdot h (\FTC(g)\cdot \FTC(h))s_*^3}{|\mysin s_* |^{\beta+2}} \dd{s_*}
 \Bigg) \notag \\ 
 &+ 2h(s)\Bigg( \frac1{L^2} \int_{\mathbb T}\int_{\mathbb T} \Gamma(s_1,s_0)^\beta \frac{ \FTC(h)(s_1,s_0)\cdot h(s_0+s_1)  s_0}  {|\mysin s_0|^\beta} \notag \\
      &-\beta  \Gamma(s_1,s_0)^{\beta+2}
\\
& \times \frac{\FTC(h)(s_1,s_0)\cdot g(s_1)(\FTC(g)(s_1,s_0)\cdot \FTC(h)(s_1,s_0))s_0^3}{|\mysin s_0 |^{\beta+2}} \dd{s_0}\dd{s_1}\notag     \\  
      &-\frac1{L^2} \int_{\mathbb T} \Gamma(s,s_*)^\beta \frac{\FTC(h)(s,s_*)\cdot h(s+s_*)  s_*}  {|\mysin s_*|^\beta}\notag  \\
      &-\beta  \Gamma(s,s_*)^{\beta+2}\frac{\FTC(h)(s,s_*)\cdot g(s)(\FTC(g)(s,s_*)\cdot \FTC(h)(s,s_*))s_*^3}{|\mysin s_* |^{\beta+2}} \dd{s_*}
 \Bigg) \notag \\
 &+ h_s(s) \Bigg( s \frac1{L^2} \int_{\mathbb T}\int_{\mathbb T} \Gamma(s_1,s_0)^\beta \frac{\FTC(h)(s_1,s_0)\cdot h(s_0+s_1)  s_0}  {|\mysin s_0|^\beta} \notag \\
        &-\beta  \Gamma(s_1,s_0)^{\beta+2}
\\
& \times \frac{\FTC(h)(s_1,s_0)\cdot g(s_1)(\FTC(g)(s_1,s_0)\cdot \FTC(h)(s_1,s_0))s_0^3}{|\mysin s_0 |^{\beta+2}} \dd{s_0}\dd{s_1} \notag \\ 
       &-\frac1{L^2} \int_{0}^s\int_{\mathbb T} \Gamma(s_1,s_0)^\beta \frac{\FTC(h)(s_1,s_0)\cdot h(s_0+s_1)  s_0}  {|\mysin s_0|^\beta} \notag \\
      &-\beta  \Gamma(s_1,s_0)^{\beta+2}
\\
& \times \frac{\FTC(h)(s_1,s_0)\cdot g(s_1)(\FTC(g)(s_1,s_0)\cdot \FTC(h)(s_1,s_0))s_0^3}{|\mysin s_0 |^{\beta+2}} \dd{s_0}\dd{s_1}  
 \Bigg).\label{eqn-h}
 \end{align}
\subsection{Evolution of $\Gamma$}  The function $\Gamma$ defined in \eqref{eqn-defn-gamma} satisfies the evolution equation
\begin{align}
    \Gamma_t = \Gamma^3 \frac{\fd\mycurve}{|\mysin s_*|}\cdot \frac{\fd\mycurve_t}{|\mysin s_*|} = \frac{\Gamma^3 s_*^2}{(\mysin s_*)^2} \FTC(g) \cdot \FTC(g_t) . \label{eqn-gamma}
\end{align}  
\subsection{Summary of dependencies}\label{subsection-summary-of-dependencies}
We now define the operators $E_1$, $E_2$, $E_3,$ and $E_4$ using the right-hand sides of the above  equations. That is, $E_1$,  $E_2$, and $E_4$  are defined as the right-hand sides of \eqref{eqn-f}, \eqref{eqn-g} and \eqref{eqn-gamma} respectively. For $E_3$ in \eqref{eqn-h} we isolate the most singular term, involving a convolution, and use $g_t=E_2[g,h,h_s,\Gamma]$ to rewrite the term $\FTC(g_t)$. This gives 
\begin{align}
    f_t &= E_1[g,h,\Gamma]\label{eqn-T1}, \\
    g_t &= E_2[g,h,h_s,\Gamma]\label{eqn-T2}, \\
    h_t &= \frac{-1}{L^\beta}  \int_{\mathbb T} \frac{h_s(s+s_*) - h_s(s_*)}{|\mysin s_*|^\beta } \dd{s_*} + E_3[g,h,h_s,\Gamma]\label{eqn-T3},  \\ 
    \Gamma_t &= E_4[\Gamma,g,F_2[g,h,h_s\Gamma]]. \label{eqn-T4}
\end{align}
The convolution term in \eqref{eqn-T3} is 
\[\mathcal H_\beta (h) :=- \int_{\mathbb T} \frac{h_s(s+s_*) - h_s(s_*)}{|\mysin s_*|^\beta } \dd{s_*}.\]
%

\begin{lem}(Skew-symmetry)
The symbol of $\mathcal H_{\beta}$ is purely imaginary. \label{lem-skew-sym}
\end{lem}
\begin{proof}
$-\mathcal H_\beta$ is  the convolution of $h_s$ with the renormalised distribution of $|\sin s|^{-\beta}$,
\[ \langle \mathcal R_{|\mysin s|^{-\beta}},\phi \rangle := \int_{\mathbb T} \frac{\phi(s)-\phi(0)}{|\mysin s|^\beta} \dd{s}. \]   
As $|\mysin s|^{-\beta}$ is even and real valued, so is the Fourier transform $F(k) =\mathcal F (\mathcal R_{|\mysin s|^{-\beta}})(k)$. Hence, the symbol of $\mathcal H_\beta$ is $ \text i  2\pi  kF(k)$, which is purely imaginary.
\end{proof}
Any operator $\text im(\del_s)$ with a purely imaginary symbol $\text im(k)$ is  skew-symmetric and so $(\text  im(\del_s)h,h)_{L^2} = 0$. We have $(\text im(\del_s)g,h)_{L^2(\mathbb T)} = (\text im(k)\hat g(k),\hat h(k))_{\ell ^2(\mathbb Z)} = - (\hat g(k), \text  im(k)\hat h(k))_{\ell ^2(\mathbb Z)}$. More importantly the operator
$\del_t -\text im(\del_s)  $ is boundedly invertible on $L^2$-based Sobolev spaces , with a solution defined via its Fourier coefficients,
\begin{align}
 \del_t f - \text i m(\del_s) f = g \iff \hat f(k) = \text e^{-\text i m(k)t} \hat f_0(k) + \int_0^t \hat g(k) \text e^{\text i m(k)(\tilde t -t) } \dd{\tilde t}.  
\end{align}
For the time dependent operator $\del_t -  L(t)^{-\beta} m(\del_s)  $, we can first perform the time rescaling $\del_{t_0} f(s,t(t_0)) = L(t)^{\beta} \del_t f$ with $t(t_0)= \int_0^{t_0} L(\tau)^{\beta} \dd{\tau}$. In these coordinates, the equation is $\partial_{t_0} f - \text i m(\partial_s) f = g$. Writing $t_0(t)$ for the inverse of $t(t_0)$, applying the above formula gives
\begin{align}
     & \del_t f - \text i  L(t)^{-\beta} m(\del_s) f = g 
     \\ \iff \hat f(k,t) &= \text e^{-\text im(k)t_0(t)} \hat f_0(k) + \int_0^{t_0(t)}  \hat g(k,t(\tilde t_0)) \text e^{\text i m(k)(\tilde t_0 -t_0(t)) } \dd{\tilde t_0}   .\end{align}
A change of variables $\tilde t_0 = t_0(\tilde t)$ gives
\begin{align}
 \hat f(k,t) &= \text e^{-\text im(k)t_0(t)} \hat f_0(k) + \int_0^{t} L(\tilde t)^{-\beta} \hat g(k,\tilde t) \text e^{\text i m(k)(t_0(\tilde t) -t_0(t)) } \dd{\tilde t}   .\label{eqn-symbol-inversion}
\end{align}

We will take advantage of this to allow the use of the abstract \CK{} theorem, despite the fact that the original contour dynamics equation involves an operator of order higher than 1. 

%
%

\section{\CK{} Theorem} 
\label{ch5section-applyCK}
In this section, we describe the abstract \CK{} theorem that we use to prove existence and uniqueness of solutions.
The version which we present below follows the notation  Sammartino and Caflisch,  \cite{sammartino1998zero} (see also\cite{safonovMR1338472}).
We have changed assumption \ref{Sammartino3} below so that the  Cauchy-type estimate remains valid for $\beta>\beta_0$. 

Because of this difference, we give the full proof below. The strategy  of the proof is based on the methods of \cite{lombardoMR2049030} and \cite{asanoMR1414217}.
After the proof, we also make some remarks about the differences with \cite{lombardoMR2049030}. 

\begin{defn}
Let $\rho_0>0$. A Banach scale $\left\{ X _ { \rho } , 0 < \rho < \rho _ { 0 } \right\}$ with norms $\|\cdot\|_\rho$ is  a collection of Banach spaces such that $X _ { \rho ^ { \prime } } \subset X _ { \rho ^ { \prime \prime } }$ with $\|\cdot \| _ { \rho ^ { \prime \prime } } \leq \|\cdot  \| _ { \rho ^ { \prime } }$ whenever $\rho ^ { \prime \prime } \leq \rho ^ { \prime } \leq \rho _ { 0 }$.
\end{defn}
 
 \begin{defn}
 Given a Banach scale $X_\rho$, $\tau>0$ and $0<\rho\le \rho_0$ and $R>0$, we define:
 \begin{enumerate}
     \item $X _ { \rho , \tau }$ to be the set of all functions $u ( t )$ from $[ 0 , \tau ]$ to $X _ { \rho }$ endowed with the norm
     \[ \| u \| _ { \rho , \tau } = \sup _ { 0 \leq t \leq \tau } \| u ( t ) \| _ { \rho }. \]
     \item $Y _ { \rho , \beta , \tau  }$ is the set of functions $u$ such that for $t\in[0,\tau]$, $u ( t )\in X_{\rho - \beta t}$,  with the norm
     \[ \| u ( t ) \| _ { \rho , \beta , \tau } = \sup _ { 0 \leq t \leq \tau } \| u ( t ) \| _ { \rho - \beta t }.\]
     \item We will denote by $X _ { \rho , \tau } ( R )$ and $Y _ { \rho , \beta , \tau } ( R )$ the balls of radius $R$ in $X _ { \rho , \tau }$ and $Y _ { \rho , \beta , \tau }$ respectively.
 \end{enumerate}
 \end{defn}
 
 \begin{thm}
 \label{thm-CK} Suppose that there exist $\rho_0>0$, $R>0, \beta_0>0,$ and $0<T<\rho_0/\beta_0,$  such that the following holds:
\begin{enumerate}[(i)]
    \myitem{(CK1)}\label{Sammartino1} For every pair $\rho,\rho'$ such that $0 < \rho' < \rho < \rho_0 - \beta_0 T$ and every $u\in X_{\rho,T}(R)$, the function $F(t,u): [0,T)  \to X_{\rho'}$ is continuous.
    \myitem{(CK2)}\label{Sammartino2} For every $\rho$ such that $0 < \rho \leq \rho _ { 0 } - \beta _ { 0 } T$,     the function $F ( t , 0 ) : [ 0 , T )   \rightarrow X _ { \rho , T } ( R )$ is continuous in $t$, and
    \[ \| F ( t , 0 ) \| _ { \rho _ { 0 } - \beta _ { 0 } t } \leq R _ { 0 } < R.\]
    \myitem{(CK3)}\label{Sammartino3}   For  any numbers $\beta \ge \beta_0$,  $s<t< \min(T,\rho_0/\beta)$, $\rho'>0$, function $\rho(s)$ such that $0 < \rho'<\rho(s) < \rho_0-\beta s $, and functions $u _ { 1 } , u _ { 2 } \in Y _ { \rho _ { 0 } , \beta, T^* } ( R )$ we have for a constant $C$ independent of $\beta$,
    \[ \left\| F \left( t , u _ { 1 } \right) - F \left( t , u _ { 2 } \right) \right\| _ { \rho ^ { \prime } } \leq C \int _ { 0 } ^ { t } \frac { \left\| u _ { 1 }(s) - u _ { 2 } (s)\right\| _ { \rho ( s ) } } { \rho ( s ) - \rho' } \dd{s}. \]
\end{enumerate}
Then there exist $\beta > \beta_0$, and $T^*\le T$  such that there is a unique $u$ belonging to $Y_{\rho_0,\beta,T^*}(R)$ that solves the equation
\[ u = F(t,u).\]
 \end{thm}
  \begin{proof} 
Let $\beta_0,\rho_0,T,R$ be fixed constants as in the theorem. 
We introduce the following weighted Banach space for $\gamma \in (0,1)$ arbitrary but fixed,  $\beta\gg 1$ to be chosen later, and $T^*=T^*(\beta):=\min(T,\rho_0/\beta)\le T$,
\begin{align}
 \mathbb S^{\gamma,\beta} &= \{ u:[0,T^*) \to X_{\rho_0 - \beta T^*} :  \|u\|^{(\gamma,\beta)} <\infty \} ,
 \end{align}
 where the weighted norm $\|u\|^{(\gamma,\beta)}$ is defined by
 \begin{align}
 \|u\|^{(\gamma,\beta)} &= \sup_{\substack{t<T^*\\ 0<\rho'<\rho_0 - \beta t} } \left(1-\frac{\beta t}{\rho_0 - \rho'}\right)^\gamma \|u(t)\|_{\rho'}. \label{defn-gammabetanorm}
  \end{align}
  Note that $\|u\|^{(\gamma,\beta)} \le \|u\|_{\rho_0,\beta,T^*}$. 
  If $0<\tilde\beta<\beta$, then $    \rho_0-\beta t < \rho_0 - \tilde \beta t$, so making the choice $\rho' = \rho_0 - \beta t$ we have
  \[ \left(1-\frac{\tilde \beta t}{\rho_0 - \rho'}\right)^\gamma  = \left(\frac{\beta - \tilde \beta}{\beta}\right)^\gamma.\] 
  This implies the following inequalities for $0<\tilde\beta<\beta$,
  \begin{align}
\|u\|^{(\gamma,\beta)} \le \|u\|_{\rho_0,\beta,T^*} \le\left(\frac{\beta}{\beta - \tilde \beta}\right)^\gamma \|u\|^{(\gamma,\tilde\beta)}. \label{beta-inclusion}
  \end{align}
%
  \subsection*{Contraction-type inequality}
  Here, we prove that for any $\beta>\beta_0$, $\gamma\in(0,1)$, and for any $u,v\in Y_{\rho_0,\beta,T^*}(R)$, we have (note that this implies $u\in \mathbb S^{\gamma,\beta}$)
  \begin{align} 
      \|F(t,u)-F(t,v)\|^{(\gamma,\beta)} \le  \frac{C  2^{1+\gamma}\|u-v\|^{(\gamma,\beta)}}{\gamma \beta}.\label{ineq-contraction}
  \end{align}
In particular, if
\begin{align}
    \beta > \frac{C 2^{1+\gamma}}\gamma ,\label{beta-restriction-1}
\end{align} then $F$ is a contraction. Define for $0<\rho'<\rho_0 - \beta s$,  $s<T^*$,
\[\rho(s) := \frac{\rho' + \rho_0 - \beta s}2 .\] 
Since $T^* < \frac{\rho_0 }\beta$,  $\rho' < \rho(s) \le \rho_0-\beta_0 s$, so we can apply \ref{Sammartino3}.  If we define $\lambda(s)$ by $\rho(s)  = \rho' + \frac{\lambda(s)}2$, i.e. 
\begin{align}
\lambda(s) := \rho_0 - \rho'-\beta s,   \label{eqn-lambdaACK}  
\end{align}
then
\begin{align} 
\rho(s) - \rho' = \frac{\lambda(s)}2 = \rho_0-\rho(s)-\beta s. \label{eqn-lambda2ACK}
\end{align}
So from \ref{Sammartino3}, we obtain for $t<T^*$,
\begin{align}
    &\|F(t,u) - F(t,v)\|_{\rho'} 
    \\
    & \le C \int_0^t \frac{\|u-v\|_{\rho(s)}}{\rho(s) - \rho'} \dd{s}
    \\
    & = C \int_0^t \frac{\|u-v\|_{\rho(s)}}{\rho(s) - \rho'}\cdot 
    \underbrace{
        \frac{(\rho_0 - \rho(s) - \beta s)^\gamma }{(\rho_0 - \rho(s))^\gamma }
    }_{
        =\big(1-\frac{\beta s}{\rho_0-\rho(s)}\big)^\gamma 
    }
     \frac{(\rho_0 - \rho(s))^\gamma } {(\rho_0 - \rho(s) - \beta s)^\gamma }\dd{s} \label{contraction-3rdline}
    \\
    &\le C (\rho_0-\rho')^\gamma \|u-v\|^{(\gamma,\beta)} \int_0^t \frac{\dd{s}}{(\lambda(s)/2)^{1+\gamma} } \label{contraction-4thline}
    \\
    &=C (\rho_0-\rho')^\gamma 2^{1+\gamma}\|u-v\|^{(\gamma,\beta)} \int_0^t \frac{\dd{s}}{(\rho_0-\rho' - \beta s)^{1+\gamma}} \label{contraction-5thline}
    \\
    &=\frac{C (\rho_0-\rho')^\gamma 2^{1+\gamma}\|u-v\|^{(\gamma,\beta)}}{\gamma \beta} \cdot \left.\frac{-1}{(\rho_0-\rho' - \beta s)^{\gamma}}  \right|_{s=0}^t \label{contraction-6thline}
    \\
    &=\frac{C  2^{1+\gamma}\|u-v\|^{(\gamma,\beta)}}{\gamma \beta} 
\Bigg(    \underbrace{ \frac{(\rho_0-\rho')^\gamma}{(\rho_0-\rho' - \beta t)^{\gamma}}
    }_{
        =\big(1-\frac{\beta t}{\rho_0-\rho' }\big)^{-\gamma} > 1 
    }
       - 1 \Bigg)
    \\
     &\le \frac{C  2^{1+\gamma}\|u-v\|^{(\gamma,\beta)}}{\gamma \beta} \cdot \frac{1}{\left (1 - \frac{\beta t}{\rho_0-\rho'}\right )^{\gamma}}.
\end{align}
In going from \eqref{contraction-3rdline} to \eqref{contraction-4thline}, we used the definition of $\|\cdot\|^{(\gamma,\beta)}$ in \eqref{defn-gammabetanorm}, $ \rho(s) >   \rho'$ and \eqref{eqn-lambda2ACK} 
 and then \eqref{eqn-lambdaACK} is used to obtain \eqref{contraction-5thline}.

Multiplying both sides by $\left (1 - \frac{\beta t}{\rho_0-\rho'}\right )^{\gamma}$ and taking a supremum over $t$ and $\rho$ with $t<T^*, \rho'<\rho_0-\beta t$ yields the desired inequality \eqref{ineq-contraction}.
\subsection*{Iteration scheme}
Set $u_0:=0$ and inductively define $u_{n} := F(t,u_{n-1})$. Then \ref{Sammartino2} implies that
\begin{align} \|u_1\|_{\rho_0,\beta,T^*} \le \|u_1\|_{\rho_0,\beta_0,T} \le R_0 < R. \label{eqn-u1-norm} \end{align}
The goal is to iteratively apply \eqref{ineq-contraction}. For this, we need to show that $u_n \in Y_{\rho_0,\beta_0,T}(R)$ for every $n>1$. This will introduce a second condition, depending on the difference $R-R_0>0$, requiring that $\beta$ must be large enough.
\subsection*{Control of norm of $u_n$}
Define the auxiliary sequence $b_k$ ($k\ge 0$)  by
\begin{align}
 b_k = \beta \left (1-\frac1{2^k}\right).  \label{eqn-beta-n}
\end{align}
Note that $b_k$ is an increasing sequence with $b_k \to \beta$. Also 
\begin{align}
    b_k \in (\beta/2,\beta). \label{eqn-b-n-ineqs}
\end{align}
Since we want to apply \eqref{ineq-contraction} (which is only valid for $\beta>\beta_0$) with $b_k$ in place of $\beta$, our construction requires
\begin{align}
    \beta > 2\beta_0 .\label{beta-restriction2}
\end{align}
Also, note that $\left(\frac{\beta}{\beta -  b_k}\right)^\gamma = 2^{\gamma k} $.  Therefore, for $k\ge 0$, by choosing $u = u_{k+1}-u_k$ in the second inequality of \eqref{beta-inclusion} and then applying \eqref{ineq-contraction} $k$ times, 
\begin{align}
 \|u_{k+1}-u_k\|_{\rho_0,\beta,T^* }
 & \le 2^{\gamma k}   \|u_{k+1}-u_k\|^{(\gamma,b_k)} \\
&\le 2^{\gamma k} \left ( \frac{C2^{1+\gamma}}{\gamma b_k}  \right)^{k} \|u_1-u_0\|^{(\gamma,b_k)} \\
&\le  \left ( \frac{C4^{1+\gamma}}{\gamma \beta }  \right)^{k} \|u_1\|^{(\gamma,b_k)}\label{ineq-gamma-norms} ,
\end{align}
since $u_0=0$, and $b_k>\beta/2$ from \eqref{eqn-b-n-ineqs}. Then \eqref{eqn-u1-norm} and the first inequality of \eqref{beta-inclusion} imply
\begin{align}
    \|u_1\|^{(\gamma,b_k)} \le \|u_1\|_{\rho_0,b_k,T^*}   \le R_0.\label{ineq-fromCK2}
\end{align} 
Applying \eqref{ineq-gamma-norms}, \eqref{ineq-fromCK2}, and \eqref{eqn-beta-n} we obtain
\begin{align}
\|u_{n}\|_{\rho_0,\beta,T^*} 
& \le \sum_{k=0}^{n-1}  \| u_{k+1}-u_{k}\|_{\rho_0,\beta,T^*}    \\
& \le  \sum_{k=0}^{n-1} \left ( \frac{C4^{1+\gamma}}{\gamma \beta }  \right)^{k} \|u_1\|^{(\gamma,b_k)}   \\
& \le  R_0 \sum_{k=0}^{\infty} \left ( \frac{C4^{1+\gamma}}{\gamma \beta }  \right)^{k}  \\
&= \frac{R_0\gamma \beta }{\gamma \beta - C4^{1+\gamma}} .
\end{align}
Therefore, in order to ensure that $u_n\in Y_{\rho_0,\beta,T^*}(R)$, we need 
\begin{align}
    \beta > \frac{C 4^{\gamma+1}}{\gamma (R-R_0)} \label{beta-restriction3}.
\end{align}
\subsection*{Existence and uniqueness of solution}
Let $\beta$ be  large enough so that
\[ \beta > \max \left(  \frac{C 2^{1+\gamma}}\gamma , 2\beta_0, \frac{C 4^{\gamma+1}}{\gamma (R-R_0)}  \right). \] Then \eqref{beta-restriction-1}, \eqref{beta-restriction2} and \eqref{beta-restriction3} are satisfied. Thus for some $R_0<R_1<R$, $u_n \in Y_{\rho_0,\beta,T^*}(R_1)$, and the contraction inequality \eqref{ineq-contraction} implies that there is a unique  solution to $u = F(t,u)$ in $Y_{\rho_0,\beta,T^*}(R)$.
 \end{proof}
 \begin{rem}
As mentioned earlier, our assumption \ref{Sammartino2} does not match the analogous assumption of \cite{sammartino1998zero}. Furthermore, while $F$ indeed satisfies a contraction-type inequality in the weighted norm $\|u\|^{(\gamma)}:=\sup_{\rho'<\rho_0-\beta t} (\rho_0-\rho'-\beta t)^\gamma \|u(t)\|_{\rho'}$ for $\beta\gg 1$  (which is what is proven in \cite{lombardoMR2049030}), it does not seem possible to control the $Y_{\rho_0,\beta,T^*}$ norm of the successive iterates $u_n$ because the right inequality of $\eqref{beta-inclusion}$ is not true for the norm $\|\cdot\|^{(\gamma)}$, since if one tries to similarly use $\rho'=\rho_0-\beta t$ to bound $\|u\|_{\rho_0,\beta,T^*}$, one finds  possible blow-up at  $t=0$. Instead, we normalise the weight $\frac{(\rho'-\rho_0-\beta t)^\gamma }{(\rho'-\rho_0)^\gamma } = \left(1-\frac{\tilde \beta t}{\rho_0 - \rho'}\right)^\gamma$.
 \end{rem}
 
 Now, we define the Banach scale that we will use in the proof.

\begin{defn}
\label{defn-Klrho-space}
Given $l\in \mathbb { N }$ and $\rho > 0$, define the open set $U_\rho\subset \mathbb C$,
\[ U_\rho := \{z \in \mathbb C : |\Im z|<\rho\}.\]
We say that a function 
$f: U_\rho \to \mathbb R^2$ is in $K ^ { l , \rho }$ if
\begin{enumerate}
    \item $f(s+\text i \tilde s)$ is analytic and 1-periodic (i.e. with period 1) in $s$.    \item For every $|\Im s | < \rho ,$ $ \partial _ { s } ^ { \alpha } f(\Re s+\text  i\Im s) \in L ^ { 2 }_{\Re s}(\mathbb T)$, that is, square integrable as a periodic function of the real
part only.
    \item The norm $\| f \| _ {K^{ l , \rho}  }$ is finite, where
\[ \| f \| _ {K^{ l , \rho} } : = \sum _ { \alpha \leq l } \sup _ { |\tilde s | < \rho } \left\| \partial _ { x } ^ { \alpha } f ( s + i \tilde s ) \right\| _ { L ^ { 2 }_s ( \mathbb T ) }. \]
\end{enumerate}
\end{defn}
\begin{defn}
Given $l\in \mathbb { N }$ and $\rho > 0$, define $U_\rho$ as in Definition \ref{defn-Klrho-space}. We say that  $f: U_\rho^2 \to \mathbb R$ is in $K^{l,\rho}_2$ if 
\begin{enumerate}
    \item $f(s+\text i\tilde s,s_*+\text i \tilde s_*)$ is 1-periodic in $ s$ and $ s_*$, and analytic in $U_\rho^2$.
    \item For every $\alpha_1 + \alpha_2 \le l $, $\max( |\tilde s |,|\tilde s_*|) < \rho $, $  \partial _ { s } ^ { \alpha_1 }\partial_{s*}^{\alpha_2} f \in L ^ { 2 }_{s,s_*}(\mathbb T^2 )$.
      \item The norm $\| f \| _ {K_2^{ l , \rho}  }$ is finite, where
\[ \| f \| _ {K_2^{ l , \rho} } : = \sum _ { \alpha_1+\alpha_2  \leq l } \sup _ { \substack{|\tilde  s | < \rho\\ |\tilde s_*|<\rho} } \left\| \| \partial _ { s } ^ { \alpha_1 }\partial _ { s* } ^ { \alpha_2 } f ( s + \text i \tilde s, s_* + i \tilde  s_* )\|_{L^2_{s}(\mathbb T)}   \right\| _ { L ^ { 2 }_{s_*} ( \mathbb T  ) }.\]
\end{enumerate}
\end{defn}
The norm $\| f \| _ {K_2^{ l , \rho} }$ can also be written as \[ \| f \| _ {K_2^{ l , \rho} } = \sum _ { \alpha_1+\alpha_2  \leq l } \sup _ { \substack{|\tilde  s | < \rho\\ |\tilde s_*|<\rho} } \left\|  \partial _ { s } ^ { \alpha_1 }\partial _ { s* } ^ { \alpha_2 } f ( s + \text i \tilde s, s_* + i \tilde  s_*  )  \right\| _ { L ^ { 2 }_{s,s_*} ( \mathbb T^2  ) }.\] 
 If $\mathcal F f(k) := \int_{\mathbb T } f(s) \text e^{-2\pi \text  i ks} \dd{s}, \ k\in\mathbb Z$ denotes the Fourier transform of $f$, then the $K^{l,\rho}$ norm is equivalent to the weighted Sobolev norm \[\| f \| _ {K^{ l , \rho} } = \| \text e^{2\pi  \rho|k|}(1+|k|^l) \mathcal F f(k)\|_{\ell^2_k( \mathbb Z)},\]which can be seen by analytic continuation in $s_0$ of the well-known identity for the Fourier transform
\[ \mathcal F_x[f(s-s_0)](k) = \text e^{-2\pi \text i s_0k } \mathcal F f(k) .\]
The following standard result follows from the Sobolev embedding theorem (see for instance \cite{evans1998partial}). 
\begin{prop}[Banach Algebra]
\begin{enumerate}
    \item For $l\ge 1$, $K^{l,\rho}$ is a Banach algebra:
\begin{align}  \|uv\|_{K^{l,\rho}} \lesssim_{l,\rho} \|u\|_{K^{l,\rho}}\|v\|_{K^{l,\rho}}.
\end{align}
    \item For $l\ge 2$, $K^{l,\rho}_2$ is a Banach algebra:
    \begin{align}  \|uv\|_{K_2^{l,\rho}} \lesssim_{l,\rho} \|u\|_{K_2^{l,\rho}}\|v\|_{K_2^{l,\rho}}.
\end{align}

\end{enumerate} 
\end{prop}

We now choose the Banach scale $X_\rho$ to apply Theorem \ref{thm-CK}. The purpose of $l$ is only to obtain the above Banach Algebra property, and is fixed to be any number $l\ge 2$. We now use the spaces $K^{l,\rho}£$ and $K^{l,\rho}_2$ to define our Banach scale. The choice is made so that a representative element of $X_\rho$ will be of the form $(\mycurve,\mycurve_s,\mycurve_{ss},\Gamma)$.
\begin{defn}
Let $l\ge 2$ be arbitrary but fixed. The Banach scale $X_\rho$ is
\[ X_\rho := K^{l,\rho} \times K^{l,\rho} \times K^{l,\rho} \times K^{l,\rho}_2. \]
\end{defn}
\subsection{Adapting the equation for Theorem \ref{thm-CK}.}

In order to apply Theorem \ref{thm-CK}, we will need to:
\begin{enumerate}[Step 1]
    \item Rewrite the  equations so that the evolution begins from zero initial data, which allows the iteration to begin using the estimate \ref{Sammartino2}.
    \item     Rewrite the evolutions in a suitable integral form that satisfies the continuity assumption \ref{Sammartino1} and the required Cauchy estimate \ref{Sammartino3}.
\end{enumerate} 
We now implement these steps.
\subsubsection*{Step 1}
Define the initial conditions 
\begin{align}
     f_0(s) &= f(s,0), \\ g_0(s) &= g(s,0),\\ h_0(s) &= h(s,0),\\ \Gamma_0(s,s_*) &= \Gamma(s,s,_*,0).
\end{align}
Then define the new variables $\tilde f,\tilde g,\tilde h, \tilde \Gamma $ by
\begin{align}
\tilde f(s,t) &= f(s,t) - f(s,0), \\ 
\tilde g(s,t) &= g(s,t) - g(s,0), \\ 
\tilde h(s,t) &= h(s,t) - h(s,0), \\ 
\tilde \Gamma(s,s_*,t) &= \Gamma(s,s_*,t) - \Gamma(s,s_*0).
\end{align}
Then the evolution equations in terms of $\tilde f,\tilde g,\tilde h, \tilde \Gamma$ are
\begin{align}
\tilde f_t &= E_1[\tilde g + g_0, \tilde h + h_0, \tilde\Gamma + \Gamma_0], \label{eqn-initial-data-form1} \\
    \tilde g_t &= E_2[\tilde g + g_0,\tilde h + h_0,\tilde h_s + h'_0,\tilde\Gamma+\Gamma_0 ],\label{eqn-initial-data-form2} \\
    \tilde h_t &= (\tilde \Gamma(0,0) + \Gamma_0(0,0) )^\beta  \mathcal H_\beta (\tilde h + h_0), \\ 
    &\quad + E_3[\tilde g + g_0 ,\tilde h+ h_0,\tilde h_s + h'_0,\tilde \Gamma + \Gamma_0], \label{eqn-initial-data-form3} \\ 
    \tilde \Gamma_t &= E_4[\tilde\Gamma+\Gamma_0,\tilde g+ g_0,F_2[\tilde g + g_0,\tilde h+ h_0,\tilde h_s + h'_0,\tilde \Gamma + \Gamma_0 ]]. \label{eqn-initial-data-form4}
\end{align}

\subsubsection*{Step 2}
The strategy is to integrate the equations for $\tilde f,\tilde g,\tilde \Gamma$ in time, and invert the operator $(\del_t - (\tilde \Gamma(0,0) + \Gamma_0(0,0))\mathcal H_\beta $ for the $h$ equation. If we define the vector of functions $\ACKu$ and initial conditions $\ACKu_0$ by
\[\ACKu(s,s_*,t):=\begin{pmatrix}
    \tilde f(s,t) \\
     \tilde g(s,t)\\ 
     \tilde h(s,t)\\ 
     \tilde \Gamma(s,s_*,t) 
\end{pmatrix},\ 
\ACKu_0(s,s_*) = \begin{pmatrix}
    f_0(s)\\
    g_0(s)\\
    h_0(s)\\
    \Gamma_0(s,s_*)
\end{pmatrix},  \]
then we can write the evolution equations as 
\begin{align}
\ACKu &= F[\ACKu], \\  F[\ACKu] &= \begin{pmatrix}
     F_1[\ACKu_0,\ACKu ]\\
     F_2[\ACKu_0,\ACKu]\\
     F_3[\ACKu_0,\ACKu, \nabla  \ACKu]\\
     F_4[\ACKu_0,\ACKu]
\end{pmatrix} , \label{eqn-abstract-form-of-evo}
\end{align}
where $\nabla \ACKu =\nabla_{s,s_*} \ACKu  $ is the spatial gradient in $s$ and $s_*$, and the component operators $ F_i$ of $F : X_\rho \to \cup_{\rho > 0} X_\rho $ are
\begin{align}
    F_1 &:= f_0 + \int_0^t E_1 \dd{t} ,\label{eqn-tildeT1} \\
     F_2 &:= g_0 + \int_0^t E_2 \dd{t}, \label{eqn-tildeT2} \\
     F_3 &:= \Big( (\del_t - (\Gamma_0(0,0,t ) + \tilde\Gamma(0,0))\mathcal H_\beta \Big)^{-1} E_3 ,\label{eqn-tildeT3} \\
     F_4 &:= \Gamma_0 + \int_0^t E_4 \dd{t}. \label{eqn-tildeT4}
\end{align}
(The omitted inputs of $F_i$ are as in \eqref{eqn-abstract-form-of-evo}, and the omitted inputs of $E_i$ are as in \eqref{eqn-initial-data-form1}, \eqref{eqn-initial-data-form2}, \eqref{eqn-initial-data-form3}, and \eqref{eqn-initial-data-form4}.) The inverse operator in \eqref{eqn-tildeT3} is defined by \eqref{eqn-symbol-inversion}. 
This completes the derivation of the equation to which Theorem \ref{thm-CK} can be applied: it  only remains to check that $F$ satisfies the assumptions of Theorem \ref{thm-CK}.

\subsection{Estimates}\label{subsect-estimates}

In this section, we give some estimates and explain how they are used to show that our system satisfies \ref{Sammartino3}. 

Many terms require a very similar approach; we will focus on  a few representative terms that illustrate the main approach. 

We write  $U_1=U_1(u),U_2=U_2(u),\dots$ to denote the images of $u$ under any of the following well behaved operators:  $U_i(u)=u$,  $U_i(u)(s,s_*) = u(s+s_*,c)$,  $U_i(u) = u(c,c_*)$, for constants $c,c_*$, or $U_i(u)=\mathcal I(u)$.  (The $c$ in $u(s+s_*,c)$  is arbitrary---the second input of  $u$ only affects $\tilde \Gamma$, and $\tilde \Gamma$ with first input $s+s_*$ doesn't appear.) For any collection of $M$ such operators $U_1,\dots,U_M$ ($M\ge 1$), we write $\mathbf U_M(u) = (U_1(u),\dots, U_M(u))$ for the function that takes values in $\mathbb C^{M'}$, where $M'$ is an integer depending on $M$ and the choices of $U_i$.

The operator $F_1[\ACKu_0,\ACKu]$ 
is a sum of
 time integrals  of products of terms of the form 
\[ \Phi(\mathbf U_M(u)(s,s_*)),\]
where $\Phi:\mathbb C^{M'} \to \mathbb C^N$ is analytic, or terms of the form  
\[ \int_{\mathbb T} \Phi(\mathbf U_M(u)(s,s_*))  \dd{\mu(s_*)}, \]
where $\mu$ is a finite measure in $s_*$, or terms of the form
\[ \int_{\mathbb T} \int_{\mathbb T} \Phi(\mathbf U_M(u)(s,s_*))   \dd{\mu(s_*)}\dd{s}.\]
For each one of these terms, we have the following elementary lemmas.

\begin{lem}[Triangle inequality for time integral]
For any function $\ACKu \in X_\rho$, $\|\int_0^t \ACKu \dd{t} \|_{\rho} \le \int_0^t \| \ACKu\|_{\rho} \dd{t}$.
\end{lem}

\begin{lem}[local Lipschitz estimates] \label{lem-analytic-is-lip}
Let 
 $\Phi(x_1,\dots,x_{M'})$ be analytic on an open set containing the set $A = \{ x: \sum_{i=1}^{M'} |x_i| \le R\} $. Then $\Phi$ is locally Lipschitz on $A$,
 \[| \Phi(x_1,\dots, x_{M'})-\Phi(y_1,\dots,y_{M'})| \lesssim_{R,\Phi} \sum_{i=1}^{M'} |x_i - y_i|. \]  For $\|\ACKu\|_\rho\le R,$ we have the estimate
\begin{align}
 \| \Phi(\mathbf U_M(u_1))-\Phi(\mathbf U_M(u_2)) \|_{K^{l,\rho}_2} \lesssim_{R,\Phi,l} \|\ACKu_1 - \ACKu_2 \|_{\rho}.
\end{align}
If in addition $\mu$ is a finite measure on $\mathbb T$, then we have the estimates
\begin{align*}      
     \left \| \int_{\mathbb T} \Phi(\mathbf U_M(u_1))(\cdot ,s_*) -\Phi(\mathbf U_M(u_2))(\cdot ,s_*)   \dd{\mu(s_*)} \right\|_{K^{l,\rho}} 
     \lesssim
      \|\ACKu_1 - \ACKu_2\|_{ \rho} ,\qquad \\
       \left | \int_{\mathbb T} \int_{\mathbb T} \Phi(\mathbf U_M(u_1))(s ,s_*) -\Phi(\mathbf U_M(u_2))(s,s_*)  \dd{\mu(s_*)} \dd{s} \right | 
     \lesssim
      \|\ACKu_1 - \ACKu_2\|_{ \rho}. \qquad 
\end{align*}
\end{lem}
In our application, we will use the finite measures  $\dd{\mu}(s_*)=\frac{s_*^{k+1} \dd{s_*}}{|\mysin s_*|^{\beta +  k}}$, for some $k>0$.

\begin{example}In $E_1$, the following terms appear (see \eqref{eqn-f}) :
\begin{align}
   \tilde E_1(u)
   &:=
   g(s)  \frac s{L^2} \int_{\mathbb T}\int_{\mathbb T} \Gamma(s_1,s_0)^\beta \frac{\FTC(h)(s_1,s_0)\cdot h(s_0+s_1)  s_0}  {|\mysin s_0|^\beta}  \dd{s_0}\dd{s _1} 
   \\
    &\quad-g(s)\frac 1{L^2} \int_{0}^s\int_{\mathbb T} \Gamma(s_1,s_0)^\beta \frac{\FTC(h)(s_1,s_0)\cdot h(s_0+s_1)  s_0}  {|\mysin s_0|^\beta}  \dd{s_0}\dd{s _1} .\end{align}
The first term is (at least away from $s=0,1$ before periodising) a product of the analytic function of $u$ and $u(0,0)$, $g(s) \frac{s}{L^2}$  with the double integral against $\frac{s_0}{|\mysin s_0|^\beta} \dd{s_0} \dd{s_1}$ of the analytic function
\begin{align}
    &\quad\ \Phi(u(s_1,s_0),\mathcal I(u)(s_1,s_0), u(s_0+s_1,s_0+s_1) ) 
    \\ &= \Gamma(s_1,s_0)^\beta \mathcal I(h)(s_1,s_0) \cdot h(s_0+s_1). 
\end{align} 
The second term is similar, and together with the first term, gives the analyticity at $s=0,1$ as well. Therefore, by  Lemma \ref{lem-analytic-is-lip}
, we have
\[  \| \tilde  E_1(u_1) - \tilde E_1(u_2) \|_{K^{l,\rho}} \lesssim  \|u_1 - u_2\|_{\rho}, \]
%
which after integrating in time, is stronger than the required estimate \ref{Sammartino3}. The other terms in $E_1$ are similar.
\end{example}

For $F_2$, most of the terms are also treated in a different way, except one which involves $h_s$. For this term, we will use the following Cauchy-type estimate. 

\begin{lem}[Cauchy Estimate]For any $l \ge 0,\rho\ge 0$, $\rho'\in(0,\rho)$, \label{lem-cauchy-estimate}
\[    
    \|\nabla  \ACKu\|_{\rho'} \le \frac C {\rho -\rho'} \| \ACKu\|_{\rho}.
\]
\end{lem}
\begin{proof}
It is enough to prove this for the component spaces $K^{l,\rho}$ and $K^{l,\rho}_2$. We give the proof for $K^{l,\rho}$, since $K^{l,\rho}_2$ can be treated in exactly the same way. That is, we shall prove for $u\in K^{l,\rho}$,
\[    \| \partial_s u\|_{K^{l,\rho'}} \leq \frac{C}{\rho-\rho'} \|u\|_{K^{l,\rho}}. \]
From the definition of the $K^{l,\rho}$ norm in Definition \ref{defn-Klrho-space}, it suffices to prove that for every $v = \partial_s^r u$, $r=0,1,\dots,l$,
\begin{align}
 \sup_{|\tilde s|<\rho'} \|\partial_s v(s+\text i \tilde s)\|_{L^2_s} \le \frac{C}{\rho-\rho'} \sup_{|\tilde s|<\rho} \| v(s+\text i \tilde s)\|_{L^2_s}. \label{just-Cauchy-for-Klrho}    
\end{align}

Set $0< \delta < \rho-\rho'$. Then the well-known Cauchy Integral Formula for a derivative gives for $z=s+\text i \tilde s\in \mathbb C$, $|\tilde s| < \rho'$,
\[ \partial_s v(s+\text i \tilde s ) = \frac{1}{2\pi \text i} \int_{|z-w|=\delta} \frac{v(w)}{(z-w)^2}\dd{w}= \frac{1}{2\pi \text i} \int_{|w|=\delta} \frac{v(s+\text i\tilde s +w)}{w^2}\dd{w}. \]
Taking the $L^2_s$ norm and using the periodicity of $v$ in the real part to obtain \eqref{used-periodicity},
\begin{align}
     \|\partial_s v(s+\text i \tilde s) \|_{L^2_s} 
     &\le \frac1{2\pi} \int_{|w|=\delta} \frac{ \|v(s+\text i  \tilde s + w)\|_{L^2_s}}{|w|^2} \dd{l(w)} \\
     &= \frac1{2\pi} \int_{|w|=\delta} \frac{ \|v(s+\text i  (\tilde s +\Im w) )\|_{L^2_s}}{|w|^2} \dd{l(w)} \label{used-periodicity} \\
     &\le \frac{1}{2\pi \delta^2}\int_{|w|=\delta} \dd{l(w)} \sup_{|\tilde{\tilde s}|<\rho'+\delta} \|v(s+\text i \tilde{\tilde s})\|_{L^2_s}\\
     & \le  \frac1{\delta} \sup_{|\tilde{\tilde s}|<\rho} \|v(s+\text i \tilde{\tilde s})\|_{L^2_s}.
\end{align}
($\dd{l}$ is the arc-length measure on $|w|=\delta$.) Taking a limit $\delta \to \rho-\rho'$, and then a supremum over all $\tilde s$ with $|\tilde s|<\rho'$ leads to \eqref{just-Cauchy-for-Klrho}. By the earlier discussion, we have finished the proof of Lemma \ref{lem-cauchy-estimate}.
\end{proof}

\begin{example}
    The first term of \eqref{eqn-g} is
    \[ \tilde E_2 (u) =  -\int_{\mathbb T} \Gamma(s,s_*)^\beta \frac{\FTC(h_s) s_*}  {|\mysin s_*|^\beta}  \dd{s_*}. \]
To show that this term satisfies the assumption \ref{Sammartino3}, first use the local Lipschitz estimates of Lemma \ref{lem-analytic-is-lip} but treating the integrand as an analytic function of $u$ and $\mathcal I(\nabla u)$. This yields a local Lipschitz estimate
\[  \|\tilde E_2(u_1) - \tilde E_2(u_2)\|_{K^{l,\rho}} \lesssim \|u_1-u_2\|_{\rho} + \|\nabla (u_1 -  u_2)\|_{\rho}. \]
Now apply the Cauchy estimate for the second term; this shows that \ref{Sammartino3} is satisfied. 
\end{example}

In a similar way, $F_4$ can be controlled by using the bounds on $F_2$, since $F_2$ appears in $F_4$.

The term $E_3$, which uses the auxillary operator used in \eqref{eqn-symbol-inversion}  to define $F_3$ is defined by \eqref{eqn-T3}. It involves the following terms where $\nabla \ACKu$ appears,
\begin{align}
    E_{31} &= \int_{\mathbb T} \FTC(\beta \Gamma^{\beta-1} \del_{s_*}\Gamma  )(s,s_*)s_*\frac{h_s (s+s_*) -  h_s(s) }{|\mysin s_*|^\beta } \dd{s_*},  \\
E_{32} &=- 2\beta\int_{\mathbb T} \Gamma(s,s_*)^{\beta+2} \frac{\FTC(h_s) (\FTC(g)\cdot \FTC(h))s_*^3}{|\mysin s_* |^{\beta+2}} \dd{s_* },\\ 
E_{33}&=- \beta\int_{\mathbb T} \Gamma(s,s_*)^{\beta+2} \frac{\FTC(h) (|\FTC(h)|^2+\FTC(g) \cdot \FTC(h_s) )s_*^3}{|\mysin s_* |^{\beta+2}} \dd{s_*}, \\
E_{34} &= h_s G ,
\end{align}
where $G$ is a collection of terms involving only $\ACKu$ and not $\nabla \ACKu$, defined by  the last four lines of \eqref{eqn-T4}. These terms are controlled by combining the above lemmas with the skew-symmetry (Lemma \ref{lem-skew-sym}) and the Cauchy-type estimate of Lemma \ref{lem-cauchy-estimate}.
\begin{thm}
    Let $z_0:\mathbb T \to\mathbb R^2$ be an analytic curve with $\Gamma_0 = \Gamma(z_0) \in L^\infty (\mathbb T^2)$. Then there exists $T^*>0,\rho_0>0$ and $\beta > 0$ such that a unique solution to \eqref{eqn-equiv-SF-system} exists in the space $u \in Y_{\rho_0,\beta,T^*}$.
\end{thm}
\begin{proof}
    Since $\mycurve_0$ is analytic on $\mathbb T$, there exists $\rho>0$ such that $\mycurve_0$ admits an analytic continuation to a complex neighbourhood  $\mathbb T +\text i  (-\rho,\rho)$ of $\mathbb T$,   that belongs to the space $K^{l,\rho}$ (recall that we have already fixed some $l\ge 2$). From $\mycurve_0$, we define the initial data  to  \eqref{eqn-abstract-form-of-evo} as
    \begin{align}
         \ACKu_0=(\mycurve_0,\del_s \mycurve_0, \del_s^2 \mycurve_0, \Gamma(\mycurve_0)) \in X_{\rho_0}.\label{eqn-abstract-init-data}
    \end{align}
    The operator $F$ in \eqref{eqn-abstract-form-of-evo} is continuous, satisfying \ref{Sammartino1} for Theorem \ref{thm-CK}, and there exists $\beta_0,T$ such that \ref{Sammartino2} is satisfied. With the estimates and earlier discussion in this subsection, \ref{Sammartino3} is satisfied, so Theorem \ref{thm-CK} applies, proving the result.
\end{proof}


\section{Acknowledgements}
Calvin Khor is supported by the studentship part of the ERC consolidator project $\text n^0$ $616797$. Jos\'e L. Rodrigo is partially supported by the ERC consolidator project  $\text n^0$ $616797$.


 \end{document}